\documentclass[reqno]{amsart}

\usepackage[T1]{fontenc}
\usepackage{microtype}
\usepackage{mathptmx}
\usepackage{graphicx}
\usepackage{algorithm}
\usepackage[noend]{algpseudocode}
\usepackage{pgfplots}

\usepackage[colorlinks=true, pdfstartview=FitV, linkcolor=black, citecolor=magenta, urlcolor=blue, pagebackref=false]{hyperref}

\usepackage{amsthm}
\newtheorem{proposition}{Proposition}
\theoremstyle{definition}

\newtheorem{remark}{Remark}

\newcommand{\1}[1]{\mathbf{1}_{#1}}

\newcommand{\Bbar}{\overline{B}}

\newcommand{\LT}[1]{\widetilde #1}

\newcommand{\N}{\mathbb{N}}

\newcommand{\RETURN}{\textbf{return}}
\newcommand{\R}{\mathbb{R}}
\newcommand{\STABLE}{\textbf{Stable}}

\newcommand{\Xbar}{\overline{X}}

\newcommand{\dto}{\downarrow}
\newcommand{\false}{\text{false}}
\newcommand{\gb}{\beta}
\newcommand{\gd}{\delta}
\newcommand{\gk}{\kappa}
\newcommand{\gl}{\lambda}

\newcommand{\pibar}{\overline{\Pi}}
\newcommand{\rmd}{{\rm d}}
\newcommand{\rme}{\mathbf{e}}
\newcommand{\rmi}{{\rm i}}
\newcommand{\sF}{\mathcal{F}}
\newcommand{\sH}{\mathcal{H}}
\newcommand{\sL}{\mathcal{L}}
\newcommand{\sS}{\mathcal{S}}
\newcommand{\sgn}{\operatorname{sgn}}
\newcommand{\sinc}{\operatorname{sinc}}

\newcommand{\true}{\text{true}}
\newcommand{\tu}{\tau(u)}

\newcommand{\whH}{\widehat{H}}
\newcommand{\whL}{\widehat{L}}
\newcommand{\whtau}{\widehat{\tau}}
\newcommand{\wh}{\widehat}

\pgfplotsset{%
    width=6.5cm,
    clean/.style={mark=none,color=black},
    filled/.style={mark=none, color=black, fill=magenta},
    standard/.style={%
        axis x line=bottom,
        axis y line=left,
        every axis x label/.style={at={(current axis.right of origin)}, xshift=20pt},
        every axis y label/.style={at={(current axis.above origin)}, right, yshift=10pt}
    },
    colormap={whitehot}{color(0cm)=(white) color(0.05cm)=(pink) color(0.5cm)=(magenta)},
    colormap={whitecold}{color(0cm)=(white) color(0.1cm)=(cyan) color(0.5cm)=(blue)}
}

\begin{document}

\title[Ruin for Tempered Stable Processes]{Finite Time Ruin Probabilities for Tempered Stable Insurance Risk Processes}
\author{Philip S. Griffin \and Ross A. Maller \and Dale Roberts}
\address{%
\begin{tabular}{lll}
Philip S.Griffin & Ross A. Maller & Dale Roberts\\
{\sf psgriffi@syr.edu} & {\sf ross.maller@anu.edu.au} & {\sf dale.roberts@anu.edu.au}\\[0.5em]
215 Carnegie Building & Mathematical Sciences Institute & Mathematical Sciences Institute \\
Syracuse University & Australian National University & Australian National University \\
Syracuse, NY 13244-1150 & Canberra ACT 0200, Australia & Canberra ACT 0200, Australia\\
\end{tabular}
}
\thanks{This work was partially supported by a grant from the Simons Foundation (\#226863 to Philip Griffin) and by ARC Grant DP1092502}

\begin{abstract}
  We study the probability of ruin before time $t$ for the family of tempered stable L\'evy insurance risk processes, which includes the spectrally positive inverse Gaussian processes. Numerical approximations of the ruin time distribution are derived via the Laplace transform of the asymptotic ruin time distribution, for which we have an explicit expression.  These are benchmarked against simulations based on importance sampling using stable processes.  Theoretical consequences of the asymptotic formulae are found to indicate some potential drawbacks to the use of the inverse Gaussian process as a risk reserve process.  We offer as  alternatives natural generalizations which fall within the  tempered stable family of processes.
\end{abstract}


\maketitle

\section{Introduction}\label{sInt}

The risk reserve of an insurance company has traditionally been modelled as a compound Poisson process with drift. In recent years more general L\'evy processes have been proposed, among them the inverse Gaussian family of processes.  Such processes have been found to approximate reasonably well a wide range of aggregate claims distributions \cite{CGT}. While the probability of eventual ruin has received a lot of attention, arguably of equal importance in  practice is the probability of ruin before some finite time horizon.
Our paper aims to study the probability of ruin before time $t$ for the inverse Gaussian family and a natural  generalisation, the tempered stable processes.

The basis of our investigation is the recent asymptotic representation, as the initial reserve grows large, of the ruin time distribution for more general ``medium-heavy'' convolution equivalent L\'evy processes \cite{G,GM2}.  This representation, via the calculation of its Laplace transform,  lends itself to a numerical approximation of the ruin time distribution, which is then benchmarked  against values obtained by simulation. Thus we are able to illustrate the use of a broad, relatively simple and computationally tractable family of processes with which to model the  risk reserve process.

We find that the asymptotic representation performs well even when the initial capital is relatively small, contrary to a view that asymptotic formulas may only be useful when the initial capital becomes extremely large.
 Additionally, the asymptotic representation provides some interesting insight with  regard to safety loading management. When a realistic safety loading is specified in the insurance risk model, we show that processes within the tempered stable family may exhibit undesirable exponential growth (in time) of the ruin probabilities, at least asymptotically. This indicates that some caution may need to be exercised in the choice of model and to aid with this task, we derive a useful relationship between the parameters to avoid an unpleasant scenario. This might have interesting implications for practitioners concerned with safety loading management.

Empirically we also observe that the asymptotic formula provides a useful lower bound for the ruin probability that can be combined with the infinite horizon ruin probability to provide a practical approximation of the true ruin probability.

\subsection{L\'evy insurance risk model}

Let $X=\{X_{t}: t \geq 0 \}$, $X_0=0$,  be a L\'{e}vy process defined on $(\Omega, \sF, P)$, with canonical triplet $(\gamma_X, \sigma_X^2, \Pi_X)$.  The characteristic function of $X$ then has the L\'{e}vy-Khintchine representation $Ee^{i\theta X_{t}} = e^{t \Psi_X(\theta)}$, where
\begin{equation}\label{lrep}
\Psi_X(\theta) = \rmi\theta \gamma_X - \tfrac{1}{2}\sigma_X^2\theta^2+ \int_{\R}(e^{\rmi\theta x}-1- \rmi\theta x \1{\{|x|<1\}})\Pi_X(\rmd x),\ {\rm for}\
 \theta \in \R.
\end{equation}

In the {\it general L\'evy insurance risk model}, the claim surplus process, which represents the excess in claims over income, is modelled by a  L\'evy process $X$ with $X_t\to -\infty$ almost surely.  Claims are represented by positive jumps, while premia and other income produce a downward drift in $X$.   The insurance company starts with a positive reserve $u$, and  ruin  occurs if this level is exceeded by $X$. The assumption $X_t\to -\infty$ a.s.\ is a reflection of the premium being set to avoid certain ruin. This setup generalises the classical Cram\'{e}r-Lundberg model in which
\begin{equation}\label{CL}
X_t=\sum_{i=1}^{N_t} U_i - pt,
\end{equation}
where the nonnegative random variables $U_i$ form an i.i.d.\ sequence with finite mean $\mu$, $N_t$ is an independent rate $\gl$ Poisson process,  and $p>\lambda\mu$.  Here $U_i$ models the size of the $i$th claim and $p$ represents the rate of premium inflow.  The assumption $p>\lambda\mu$ is the \emph{net profit condition} needed to ensure that $X_t\to -\infty$ a.s.
See \cite{Ar2} for background.



\subsection{The convolution equivalent model}
\label{sub:the_convolution_equivalent_model}

A natural class which includes the tempered stable distribution and the inverse Gaussian distribution is the class of \emph{convolution equivalent distributions}. Definitions and basic results for  convolution equivalent distributions and the corresponding  convolution equivalent L\'evy insurance risk processes are set out in  detail in  Kl\"{u}ppelberg, Kyprianou and Maller \cite{kkm} and Griffin and Maller \cite{GM2}, and associated papers, so we only outline the main ideas here.  A comparison of the medium heavy convolution equivalent condition, the light-tailed  Cram\'er condition ($E e^{\nu_o X_1} = 1$ for some $\nu_0 > 0$) and the heavy tailed subexponential condition can also be found in \cite{GM2}.

Denote the class  of (non-negative) convolution equivalent distributions of index $\alpha>0$ by ${\sS}^{(\alpha)}$. A L\'evy process is said to be {\it convolution equivalent}\footnote{See Borovkov and Borovkov \cite{BandB} and Foss, Korshunov and Zachary \cite{FKZ} for further background on subexponential and convolution equivalent distributions.}, written
\begin{equation}\label{c1}
X_1^+\in {\sS}^{(\alpha)}\ {\rm for\ some}\ \alpha >0,
\end{equation}
if the distribution of $X_1^+$ is in $\sS^{(\alpha)}$ for some $\alpha > 0$.
The \emph{convolution equivalent L\'evy insurance risk model} is one in which
\begin{equation}\label{c3}
X_1^+\in {\sS}^{(\alpha)}\ \text{for some}\ \alpha >0\ \text{ and }\ X_t\to -\infty\ \ a.s.
\end{equation}

Membership of ${\sS}^{(\alpha)}$, by definition, is a property of the positive tail of the distribution of $X_1$. Condition \eqref{c1} can equivalently be expressed in terms of the positive tail $\pibar_X^+(u)=\Pi_X((u,\infty))$ of the L\'evy measure
(see \cite{kkm}).
Assuming $\pibar_X^+(x_0)>0$ for some $x_0>0$, so that $X$ has positive jumps with probability 1, we say that $\pibar_X^+\in {\sS}^{(\alpha)}$ if the same is true of the corresponding renormalised tail $(\pibar_X^+(\cdot)/\pibar_X^+(x_0))\wedge 1$.
With this understanding, \eqref{c1} is equivalent to
\begin{equation}\label{c2}
\overline{\Pi}^+_{X} \in {\sS}^{(\alpha)}\ {\rm for\ some}\ \alpha >0.
\end{equation}

Convolution equivalent distributions of index $\alpha$ have exponential moments of order $\alpha$, but of no larger orders.  Thus, if $\psi_X$ denotes the cumulant of $X$, so that
\begin{equation*}
Ee^{\beta X_t} = e^{t\psi_X(\beta)},
\end{equation*}
then $\psi_X(\beta)$ is finite if and only if $\beta\le \alpha$.

Some asymptotic aspects of the model \eqref{CL} where $U_1$ has a convolution equivalent distribution were recently considered by Tang and Wei \cite{TangWei10}. In particular, explicit asymptotic formulas for the Gerber-Shiu function in the infinite horizon case were derived.
Theoretical and numerical comparisons between models under the Cram\'er condition or a convolution equivalent condition were recently carried out in \cite{GMvS12} for general L\'evy insurance risk processes. It was observed that the ``medium-heavy'' regime transitions continuously into the ``light-tailed'' Cram\'er regime as certain parameters describing the models are varied. The convolution equivalent model was suggested as providing a broad and flexible apparatus for modelling the insurance risk process.

\subsection{Eventual ruin}

Convolution equivalent L\'evy processes were introduced into risk theory in  \cite{kkm}.  In addition to \eqref{c2}, \cite{kkm} assumed
\begin{equation}\label{Eless1}
Ee^{\alpha X_1}< 1.
\end{equation}
Condition \eqref{Eless1} implies that $(e^{\alpha X_t})_{t\ge 0}$ is a non-negative supermartingale from which it follows that $X_t\to -\infty$ a.s., so the second condition in \eqref{c3} is automatic in this case.

For a given initial reserve $u>0$, the \emph{ruin time} is defined by
\begin{equation}\label{tau}
\tau(u)= \inf \{t \geq0 : X_{t}>u \}.
\end{equation}
The main results in \cite{kkm} include the following asymptotic estimate for the probability of eventual ruin. Assume \eqref{c2} and  \eqref{Eless1}. Then
\begin{equation}\label{tult}
\lim_{u\to\infty} \frac{P(\tau(u)<\infty)}{\overline{\Pi}^+_{X}(u)}=\frac {Ee^{\alpha \Xbar_\infty}}{-\psi_X(\alpha)},
\end{equation}
where
\begin{equation}\label{GXbar}
\Xbar_t=\sup_{0\le s\le t} X_s.
\end{equation}
This expression for the limit differs in form from that given in \cite{kkm}, but is equivalent; see Remark~\ref{rem1}.  Under  \eqref{Eless1}, $\psi_X(\alpha)<0$ and $Ee^{\alpha \Xbar_\infty}<\infty$.  If $Ee^{\alpha X_1}\in [1,\infty)$ then $Ee^{\alpha \Xbar_\infty}=\infty$,  but  $Ee^{\alpha \Xbar_t}<\infty$ for all $t\ge 0$; see Lemma 2.1 in \cite{G}.

\subsection{Ruin in finite time} \label{sub:finite_time}

A more difficult problem than the probability of eventual ruin is to find the distribution of the ruin time itself. For convolution equivalent processes, partial results in this direction  were obtained by
Braverman \cite{br}, Braverman and Samorodnitsky \cite{BS}, and Albin and Sund\'en \cite{AS}.\footnote{Heavy tailed (subexponential processes) are treated in Asmussen and Kl\"{u}ppelberg \cite{AK}. For the light-tailed ``Cram\'er case", see \cite{Ar2}.}
 More recently,  the following explicitly defined asymptotic estimate was obtained in Griffin \cite{G}  and Griffin and Maller \cite{GM2} under the sole assumption \eqref{c2}:
\begin{equation}\label{rft}
P(\tau(u)\le t)= \overline{\Pi}^+_{X}(u)B(t) +o({\overline{\Pi}^+_{X}(u)})\ \text{ a.s. } u\to\infty,
\end{equation}
where the function $B(t)$ satisfies
\begin{equation}\label{H}
B(t)=\int_0^t
e^{\psi_X(\alpha) s}Ee^{\alpha \Xbar_{t-s}}\ d s.
\end{equation}
Under Condition \eqref{Eless1}, the estimate in \eqref{rft} is uniform in $t\ge 0$,  $B$ is bounded, and by monotone convergence,  $B(t)$ increases as $t\to\infty$ to
\begin{equation}\label{Binf}
B(\infty)=\frac {Ee^{\alpha \Xbar_\infty}}{-\psi_X(\alpha)}
\end{equation}
which coincides with the limit in \eqref{tult}.  In this case the estimate may be rewritten in a more  intuitively appealing form. From \eqref{tult},  \eqref{rft} and \eqref{Binf}, we have for $t>0$
\begin{equation}\label{tulta}
P(\tau(u)\le t)={P(\tau(u)<\infty)}\left(\frac{B(t)}{B(\infty)}+o(1)\right)\ \text{ as } u\to\infty.
\end{equation}
Thus, asymptotically as $u\to\infty$, $P(\tau(u)\le t)$ factors as the product of the probability of eventual ruin and a distribution function in $t$ given by $B(t)/B(\infty)$.  Moreover this estimate is  uniform in $t\ge 0$.

When $Ee^{\alpha X_1}\ge 1$, \eqref{rft} provides an estimate which is uniform on compact sets in $t\ge 0$, but the function $B$ is unbounded and the limit in \eqref{tult} is infinite. Thus $P(\tau(u)\le t)$ is no longer proportional to $P(\tau(u)<\infty)$ and
\eqref{tulta} does not hold.

\subsection{Overview}\label{sub:overview}

The quite explicitly defined form of $B(t)$ in \eqref{H}  opens the possibility of calculating it numerically for an appropriate class of models, with the hope that the estimates may be used for guidance in some real-life modelling situations so as to derive useful information about the ruin time distribution. But in  order to implement this program, a number of questions must be addressed. First, what can be said about properties of $B$? Second, how do we obtain a good numerical approximation for the expression given by \eqref{H}? Third, once numerical results are at our disposal, how well do the approximations \eqref{rft} and \eqref{tulta} perform compared to, say, a direct simulation of the ruin time probabilities for different values of $u$ and $t$?  The aim of this paper is to give answers to these questions.

\section{Some Fluctuation Theory}\label{propofF}

In order to investigate properties of $B$ we need to introduce some notation and a few basic results from the fluctuation theory of L\'evy processes as set out in
Bertoin \cite{bert}, Sato \cite{sato} and Kyprianou \cite{kypbook}.

\subsection{Inverse local-time and ladder height processes}
\label{sub:inverse_local_time_and_ladder_height_processes}

Let $(L^{-1}_t,H_t)_{t \geq 0}$ denote the bivariate ascending inverse local time and ladder height subordinator process of $X$.
 The  bivariate descending inverse local time and ladder height subordinator
  is denoted by $(\widehat{L}_t,\widehat{H}_t)_{t\geq0}$.
  Their Laplace exponents $\kappa(a,b)$ and $\wh \kappa(a,b)$
are defined, for values of $a,b\in \R$ for which the expectations are finite, by
\begin{equation} \label{kapdef}
e^{-\kappa(a,b)} = E (e^{-aL^{-1}_1 -bH_1};1<L_\infty)
\quad {\rm and}\quad
e^{-\wh \kappa(a,b)} =E(e^{-a\wh L^{-1}_1-b\wh H_1};1<\wh L_\infty).
 \end{equation}
The random variables $L_\infty$ and $\wh L_\infty$ are exponentially distributed  with parameters $q\ge 0$ and $\wh q\ge 0$ respectively, with the understanding  that if $q$ or $\wh q$ is zero then the resulting random variable is identically infinite.
  We can write
\begin{equation} \label{kapexp}
\kappa(a,b) = q+\rmd_{{L}^{-1}}a+\rmd_{H}b+\int_{t\ge0}\int_{h\ge0} \left(1-e^{-at-bh}\right) \Pi_{{L}^{-1}, {H}}(\rmd t, \rmd h),
\end{equation}
where $\rmd_{{L}^{-1}}\ge 0$ and $\rmd_{H}\ge 0$ are drift constants, and
$\Pi_{{L}^{-1}, {H}}(\rmd t, \rmd h)$ is the bivariate L\'evy measure of $(L^{-1},H)$.
 Similarly,
 \[
\wh \kappa(a,b) =\wh q+\rmd_{{\wh L}^{-1}}a+\rmd_{\wh H}b
+\int_{t\ge0}\int_{h\ge0} \left(1-e^{-at-bh}\right)
\Pi_{{\wh L}^{-1}, {\wh H}}(\rmd t, \rmd h).
\]
These integrals are finite at least for $a\ge0$, $b\ge0$.
We denote the marginal  L\'evy measures of $L^{-1}$ and $H$ by
$\Pi_{L^{-1}}(\rmd t)$ and $\Pi_{H}(\rmd h)$, and similarly for the corresponding hat quantities.

When $\lim_{t\to \infty} X_{t} = -\infty$ a.s.\ the increasing ladder process
$(L^{-1}_t,H_t)_{t \geq 0}$ is defective.  In that case $(L^{-1},H)$ is obtained from  a non-defective bivariate subordinator $({\sL}^{-1},{\sH})$ by independent exponential killing with rate $q > 0$. The decreasing ladder process
 $(\widehat{L}_t,\widehat{H}_t)_{t\geq0}$ is proper when $X_{t} \to -\infty$ a.s.,
 and we then have $\wh q=0$.

The Wiener-Hopf factors of $X$ may be expressed in terms of these Laplace exponents.  In particular
\begin{equation}\label{WHk}
Ee^{-a \Xbar_\rme}=\frac{\kappa(\gd,0)}{\kappa(\gd,a)}
\end{equation}
where $\rme$ is independent of $X$ and exponentially distributed with parameter $\gd$, and $a\ge 0$.  If  $\pibar^+_X\in{\sS}^{(\alpha)}$, then $Ee^{\alpha X_1}<\infty$, $Ee^{\alpha H_1}<\infty$ and \eqref{WHk} remains true for
$a \ge -\alpha$.

The Wiener-Hopf factorization involves an arbitrary constant which we may take to be one by choice of normalization of the local times.  In other words we assume
\begin{equation}\label{WHF}
-\log Ee^{i\theta X_1}=[-\log Ee^{i\theta H_1}][-\log Ee^{-i\theta \whH_1}].
\end{equation}

\subsection{The spectrally positive case}\label{sSP}

When $X$ is spectrally positive, that is, $\Pi_X((-\infty,0))=0$,
more explicit expressions are available for $\gk$ and $\wh\gk$.
In this situation we  take $\whL_t=-\inf_{0<s\le t}X_s$, so the inverse process $\whL_y^{-1}$ is the passage time subordinator
\begin{equation}\label{whtauy}
\whtau_y:=\inf\{t>0:\inf_{0<s\le t}X_s<-y\},\ y\ge 0,
\end{equation}
and $\whH_t=t$ on $\{\whtau_t<\infty\}$.
Let
\begin{equation}\label{Phi}
\Phi_X(\delta)= \inf\{\beta:\psi_X(\beta)=\delta\}, \ \delta\ge 0.
\end{equation}
Since $\psi_X$ is strictly decreasing on $(-\infty, \Phi_X(0)]$, the function $\Phi_X:[0,\infty)\to (-\infty, \Phi_X(0)]$ is the inverse of the restriction of
$\psi_X$  to $(-\infty, \Phi_X(0)]$.  It follows that
\begin{equation}
\label{kapphipsi1}
\wh\kappa_X(\delta,\beta) = \beta-\Phi_{X}(\delta), \quad \delta\ge 0, \beta\in \R,
\end{equation}
and as a consequence of the choice of normalisation in \eqref{WHF},
\begin{equation}
\label{kapphipsi2}
\kappa_X(\delta,\beta)= \frac{\psi_X(-\beta)-\delta}{\beta+\Phi_{X}(\delta)},
 \quad \delta\ge 0, \beta\ge 0.
\end{equation}
See Section 8.1 of \cite{kypbook} or Section VII.1 of \cite{bert} which, note, both apply to spectrally negative processes.
If in addition $\pibar^+_X\in{\sS}^{(\alpha)}$, then \eqref{kapphipsi2} remains true for
$\delta\ge 0, \beta\ge -\alpha$.

\section{Properties of $B$ and consequences for the insurance risk process}\label{propofB}

We now present some analytical properties of the function $B$ and discuss their implications for the L\'evy insurance risk process.

\subsection{Laplace transform and rate of growth of $B$}\label{RGLT}

Direct analytic evaluation of $B$ through \eqref{H} is not feasible so we turn to evaluating it by numerically inverting its Laplace transform.  The following proposition provides the required theoretical result.

\begin{proposition}\label{LTofB}
Assume \eqref{c2} holds.  Then for $\gd > \psi_X(\alpha)\vee 0$,
\begin{equation} \label{HLT}
\LT{B}(\gd) := \int_0^\infty e^{-\gd t}B(t)\rmd t = \frac{\kappa(\gd, 0)}{\gd(\gd-\psi_X(\alpha))\kappa(\gd,-\alpha)}.
\end{equation}
When, further,  $X$ is spectrally positive, this takes the form
\begin{equation}\label{HLTspecpos}
	\LT{B}(\gd) = \frac {\Phi_X(\gd)-\alpha}{(\gd-\psi_X(\alpha))^2\Phi_X(\gd)}.
\end{equation}
\end{proposition}

\begin{proof}
Substituting for $B$ from \eqref{H} we obtain
\begin{equation}
\begin{aligned}
\gd\int_0^\infty e^{-\gd t}B(t)\,\rmd t
&= \gd \int_0^\infty e^{-\gd t}\int_0^t e^{\psi_X(\alpha)s} E e^{\alpha \Xbar_{t-s}}\,\rmd s\,\rmd t\\
&= \gd \int_0^\infty e^{-\gd t}e^{\psi_X(\alpha)t} \int_0^t e^{-\psi_X(\alpha)s} E e^{\alpha \Xbar_s}\,\rmd s\,\rmd t\\
&=  \gd\int_0^\infty \int_s^\infty e^{-(\gd -\psi_X(\alpha)) t}\,\rmd t\, e^{-\psi_X(\alpha)s} E e^{\alpha \Xbar_s}\,\rmd s\\
&= \frac{1}{\gd - \psi_X(\alpha)} \int_0^\infty  E e^{\alpha \Xbar_s}\gd e^{-\gd s}\,\rmd s\\
&= \frac{1}{\gd - \psi_X(\alpha)} E e^{\alpha \Xbar_\rme}
\end{aligned}
\end{equation}
where $\rme$ is distributed as exponential with parameter $\gd$ independently of $X$. Thus \eqref{HLT} follows from \eqref{WHk}, and then \eqref{HLTspecpos} from \eqref{kapphipsi2}.
\end{proof}

Next we investigate how \eqref{Eless1} and complementary conditions relate to the growth of $B$.  This will have potential consequences for
modelling the insurance risk process.

\begin{proposition}\label{GrofB}
Assume \eqref{c2} holds.

(i)\ If $Ee^{\alpha X_1} <1$ then
\begin{equation}\label{Bn}
B(\infty)=\frac{q}{-\psi_X(\alpha)\kappa(0,-\alpha)}\in(0,\infty).
\end{equation}

(ii)\ If $Ee^{\alpha X_1} > 1$ then
\begin{equation} \label{GrC}
\lim_{t\to\infty} \frac{\ln B(t)}{t}=\psi_X(\alpha).
\end{equation}

(iii)\ If $Ee^{\alpha X_1} = 1$ then \eqref{GrC} continues to hold
(in which $\psi_X(\alpha)=0$),  together with
\begin{equation} \label{gg}
\liminf_{t\to\infty} \frac{B(t)}{t}\ge 1;
\end{equation}
thus, $B(\infty)=\infty$.
If in addition $EX_1e^{\alpha X_1}<\infty$, then
\begin{equation} \label{GrC1}
\limsup_{t\to\infty} \frac{B(t)}{t^2}<\infty.
\end{equation}
\end{proposition}

\begin{proof}
Assume $Ee^{\alpha X_1} <1$. In that case $\psi_X(\alpha)<0$, $\kappa(0, 0)=q > 0$ and $\kappa(0,-\alpha)>0$; see Proposition 5.1 of \cite{kkm}. First integrate  by parts and then use \eqref{HLT} to obtain
\begin{equation*}\begin{aligned}
\int_0^\infty e^{-\gd t}B(\rmd t)=\int_0^\infty \gd e^{-\gd t}B(t)\rmd t
=\frac{\kappa(\gd, 0)}{(\gd-\psi_X(\alpha))\kappa(\gd,-\alpha)},\ \ \gd>0.
\end{aligned}\end{equation*}
Letting $\gd\downarrow 0$ then gives \eqref{Bn}.

Now assume $Ee^{\alpha X_1} \ge 1$. Observe  that $\ln Ee^{\alpha \Xbar_{s}}$ is subadditive,
hence by Fekete's lemma,
\begin{equation}\label{r}
\lim_{s\to\infty}\frac{\ln Ee^{\alpha \Xbar_{s}}}{s}=r
\end{equation}
for some $r<\infty$.  Since  $\ln Ee^{\alpha \Xbar_{s}}\ge \ln Ee^{\alpha X_{s}}=s\psi_X(\alpha)$, it follows that
$r\in [\psi_X(\alpha),\infty)$.  If $r>\psi_X(\alpha)$, choose $\gd\in (\psi_X(\alpha), r)$.
It follows easily from \eqref{H} that
\begin{equation}
\liminf_{t\to\infty} \frac{\ln B(t)}{t}>\gd.
\end{equation}
But then $\LT{B}(\gd) =\infty$ which contradicts \eqref{HLT}, so $r=\psi_X(\alpha)$.

Finally assume $Ee^{\alpha X_1} = 1$.  Then $\psi_X(\alpha)=0$ so \eqref{gg}
is immediate from \eqref{H}.
In general when $ \pibar_X^+\in {\sS}^{(\alpha)}$, we know  only that $Ee^{\alpha X_{1}}<\infty$, but now assume further that
$
E(X_1e^{\alpha X_{1}})<\infty.
$
Since $e^{\alpha X_t}$ is a (sub)martingale, by Doob's $L^1$-maximal inequality (see Exercise 5.4.6 of \cite{Dur}),
\begin{equation}\label{L1}
Ee^{\alpha \Xbar_n}\le (1-e^{-1})\left(1+E(\alpha X_n^+e^{\alpha X_n})\right).
\end{equation}
Now
\begin{equation}
X_n^+e^{\alpha X_n}\le \sum_{i=1}^n(\Delta X_i)^+\Pi_{j=1}^ne^{\alpha \Delta X_j}
\end{equation}
where $\Delta X_i=X_i-X_{i-1}$.  Thus
\begin{equation}\label{L2}
E(X_n^+e^{\alpha X_n})\le nE\left((\Delta X_1)^+e^{\alpha \Delta X_1}\right)(Ee^{\alpha \Delta X_1})^{n-1}.
\end{equation}
Hence substituting into \eqref{L1} and using monotonicity, for some constant $C<\infty$ and all $s\ge 0$
\begin{equation}\label{Doob}
Ee^{\alpha \Xbar_{s}}\le C(1+s)e^{\psi_X(\alpha)s}=C(1+s),
\end{equation}
from which \eqref{GrC1} follows after substitution in \eqref{H}.
\end{proof}

\begin{remark}\label{rem1}
 {\rm  From  \eqref{Binf}, the limit in \eqref{tult} may be alternatively expressed as in \eqref{Bn}.  This is the form of the limit given in \cite{kkm}; see Theorem 4.1 and Proposition 5.3 therein.
The assumption $EX_1e^{\alpha X_1}<\infty$ arises in connection with Cram\'er's large deviation estimate for the probability of eventual ruin in the $\psi_X(\alpha)=0$ case; see \cite{BD}.}
\end{remark}

If $Ee^{\alpha X_1} > 1$ then \eqref{GrC} shows that $B(t)$ grows exponentially with $t$. Thus for appropriately large $u$ and $t$, increasing the time horizon by one unit increases  the probability of ruin by a factor of $e^r$, where $r$ is essentially $\psi_X(\alpha)$.
For example if $r=3$ this is a factor of at least 20. This is clearly a situation which would concern any insurance company.

If, instead,  $Ee^{\alpha X_1}= 1$ we are in the realm of the classical Cram\'er condition where $B$ grows subexponentially.
With the additional mild assumption $EX_1e^{\alpha X_1}<\infty$, $B$ grows at most quadratically, as shown by \eqref{GrC1}.  However $B$ is still unbounded and the estimate in \eqref{tult} is not uniform over all $t\ge 0$ in this case. Further, the  probability of eventual ruin is  of a different order to the probability of ruin in finite time.

If $Ee^{\alpha X_1} < 1$ then none of the above issues arise.  $B$ is bounded, the estimate is uniform in $t$ and the finite ruin time probabilities are comparable to the infinite horizon ruin probabilities.  Furthermore the modified form \eqref{tulta} of the limit holds, which as we will demonstrate provides a superior estimate for small $u$.

In conclusion, unless exponential growth of the finite horizon ruin probabilities is to be modelled, then it is necessary that the claims surplus process satisfy $Ee^{\alpha X_1}\le 1$.  Within this class, it may be desirable to further restrict to $Ee^{\alpha X_1}< 1$
due to the intuitive appeal,  and uniformity in $t$, of the estimate in \eqref{tulta}.

\subsection{Moments and smoothness of $B$}\label{LTMB}

In this section we give some subsidiary results which expand on the properties of $B$.
Since it will not be used in the remainder of the paper in an essential way, the proof of Proposition~\ref{Bprop} is deferred to an Appendix.

\begin{proposition}\label{Bprop}
Assume   $X$ is spectrally positive and has no Brownian component, \eqref{c2} holds, and $Ee^{\alpha X_1}<1$
(so that $\lim_{t\to\infty}X_t=-\infty$ a.s.).
Then $\int_0^\infty t B(dt)<\infty$; thus, the limit distribution corresponding to $B$ has finite expectation.
\end{proposition}

To conclude this section, we mention some smoothness properties of $B$. Rewrite \eqref{H} as
\begin{equation}\label{H2}
B(t)=e^{\psi_X(\alpha) t}
\int_0^t e^{-\psi_X(\alpha) s}Ee^{\alpha \Xbar_{s}}\rmd s.
\end{equation}
From this we see that $B$ has a density $B'$ satisfying
\begin{equation}\label{bd}
B'(t)=\psi_X(\alpha)B(t) +Ee^{\alpha \Xbar_t},\ t\ge 0.
\end{equation}
If $\psi_X(\alpha)=0$, it follows immediately that $B'(t)\to1$ as $t\to 0$ and so $B(t)$ increases approximately linearly near 0.  The same conclusion holds when $\psi_X(\alpha)\neq0$.  For this  first observe that from
\eqref{H2},
\begin{equation}\label{H2in}
\frac{e^{\psi_X(\alpha) t}-1}{\psi_X(\alpha)}\le
B(t)\le \frac{e^{\psi_X(\alpha) t}-1}{\psi_X(\alpha)}Ee^{\alpha \Xbar_t}.
\end{equation}
Hence by \eqref{bd},
\begin{equation}\label{bd2}
\frac{e^{\psi_X(\alpha) t}-1+Ee^{\alpha \Xbar_t}}{\psi_X(\alpha)}\le \frac{B'(t)}{\psi_X(\alpha)}\le \frac{e^{\psi_X(\alpha) t}Ee^{\alpha \Xbar_t}}{\psi_X(\alpha)}.
\end{equation}
Thus again $B'(t)$  tends to 1 as $t\to 0$.

\section{Tempered Stable Processes}\label{PCEP}

In this section we set out a  parametric class of tempered stable L\'evy processes which will be used  as the basis for later calculation and simulations.
All processes will be spectrally positive, i.e., have no downward jumps, but initially are not required to drift to $-\infty$, so we denote them by $Y$ to distinguish them from insurance risk processes which we will continue to denote by $X$.  $X$ will be obtained from $Y$ in Section~\ref{TSmodel} by subtracting a drift.

\subsection{Tempered stable processes}\label{sTSP}
Let $Y$ be a L\'evy process with characteristic triplet $(\gamma_Y, \sigma_Y^2, \Pi_Y)$, where
\begin{equation}\label{tempstabletriplet}
\gamma_Y= c(1-\rho)^{-1}-c\int_0^1(1-e^{-\alpha x})\frac {\rmd x}{x^\rho},\qquad
\sigma_Y^2=0, \qquad
 \Pi_Y(\rmd x)=\frac{ce^{-\alpha x}\ \rmd x}{x^{1+\rho}},\ \ x>0.
\end{equation}
Then $Y$ is a tempered stable process with parameters $c>0$, $\alpha >0$ and $\rho\in(0,1)\cup (1,2)$.
The choice of $\gamma_Y$ is made so that the cumulant  of $Y$ is
\begin{equation}
\label{psits}
\psi_Y(\theta)=\ln Ee^{\theta Y_1} =-c\Gamma(-\rho)[\alpha^\rho-(\alpha-\theta)^\rho], \quad \theta\le \alpha,
\end{equation}
where $\Gamma$ denotes the usual gamma function.
From this we find
\begin{equation}
\label{meanX}	
EY_1=\psi_Y'(0)= -c\rho\Gamma(-\rho)\alpha^{\rho-1}.
\end{equation}
$Y$ is a pure jump subordinator when $\rho\in(0,1)$ while for $\rho\in (1,2)$ it is spectrally positive but of unbounded variation.

\subsection{Inverse Gaussian processes}\label{sub:inverse_gaussian}

Suppose $Y$ has characteristics given by \eqref{tempstabletriplet} specialized by taking $\rho = 1/2$.  Then $Y$ is an \emph{inverse Gaussian process}.  It is a pure jump subordinator with cumulant
\begin{equation}\label{invg}
	\psi_Y(\theta) = 2c\sqrt{\pi}(\sqrt{\alpha} - \sqrt{\alpha-\theta}),\
	\theta\le \alpha,
\end{equation}
and mean
\begin{equation}\label{mean}
EY_1=c\sqrt{\frac{\pi}{\alpha}}.
\end{equation}

\section{Tempered Stable Insurance Risk Model}
\label{sec:parametrisation}

We turn now to insurance risk modelling based on a L\'evy process $X$, obtained from the tempered stable process $Y$ of Section~\ref{sTSP}, by subtracting a drift.  We will focus on the case $\rho\in (0,1)$.  We continue to require $X_t \to - \infty$ a.s.\ to reflect that the insurance company intends to collect sufficient premiums to avoid certain ruin.
A sufficient condition that ensures this is $E e^{\alpha X_1} \le 1$.
We note however that  $X_t \to - \infty$ a.s.\ may still hold when $E e^{\alpha X_1} > 1$.
This has some interesting consequences for insurance model parametrisation and safety loading management.

\subsection{Aggregate claims and the claims surplus model} \label{TSmodel}

Let $Y$ be the tempered stable process of Section~\ref{sTSP} with $\rho\in (0,1)$.  In that case $Y$ is a pure jump subordinator  which we can take to model the aggregate claims process.  We may then consider a one parameter family of claims surplus processes indexed by the premium rate $p$:
\begin{equation}\label{Xp}
X^{(p)}_t=Y_t-pt.
\end{equation}
We will use a superscript $p$ to denote quantities computed from $X^{(p)}$, for example its cumulant $\psi^{(p)}_X(\theta)=\psi_Y(\theta)-p\theta$.
By the strong law for L\'evy processes,
\begin{equation}\label {SLLN}
X^{(p)}_t\to -\infty\ \ {\rm a.s.} \iff\ \ p>EY_1.
\end{equation}
Note also that by \eqref{psits} and \eqref{meanX}
\begin{equation}\label{less1cond}
Ee^{\alpha X^{(p)}_1}< 1\ \iff\ \ \psi_Y(\alpha)< p\alpha\ \iff\ \  p> -c\Gamma(-\rho)\alpha^{\rho-1}\ \iff\ \  p>\frac{EY_1}{\rho},
\end{equation}
thus confirming that $Ee^{\alpha X_1^{(p)}}< 1$ implies $X^{(p)}_t\to -\infty$ a.s., via \eqref{SLLN}, since $\rho\in(0,1)$.

Taking $\rho = 1/2$ we obtain a general model where aggregate claims are modelled by an inverse Gaussian process. The use of the inverse Gaussian process in this context is discussed in Garrido and Morales \cite{GaMo}, Morales \cite{Mor} and  Chaubey, Garrido, and Trudeau \cite{CGT} among others. The choice $\rho = 1/2$ makes a number of computations easier. For example, it is a tedious but simple matter to confirm that $\Phi^{(p)}_X$, defined in \eqref{Phi}, is given by
\begin{equation*}
\Phi_X^{(p)}(\delta)=\frac{2\pi c^2 +2\sqrt{\pi} c\left(\sqrt{(\sqrt{\alpha}p-\sqrt{\pi}c)^2+\delta p}-\sqrt{\alpha}p\right)+\delta p}{-p^2}.
\end{equation*}

\subsection{Safety loading}
\label{sub:safety_loading}

In the Cram\'er-Lundberg model \eqref{CL}, if we write
\begin{equation}\label{sl}
p=(1+\xi)\gl\mu,
\end{equation}
then $\xi$ is called the safety loading and, in practice, its value is typically of order $0.2$ \cite{grandell}.

For the tempered stable model of Section~\ref{TSmodel}, the natural interpretation of the safety loading is to write
\begin{equation}\label{sl01}
p=(1+\xi)EY_1.
\end{equation}
Thus by \eqref{less1cond},
\begin{equation}\label{rhogt}
\psi^{(p)}_X(\alpha)<0 \ \iff\ \ \ \rho>\frac 1{1+\xi} \ \iff\ \ \ \xi>\frac{1-\rho}{\rho}.
\end{equation}
It is interesting to note that this condition imposes no restrictions on the parameters $c$ and $\alpha$.  This will be further elaborated on in the next section.

It has been suggested that the inverse Gaussian process be used as a model for aggregate claims.  However when $\rho=1/2$, $\xi=0.2$ and $p$ is given by $\eqref{sl01}$, we are forced by \eqref{rhogt} to consider the situation in which $\psi^{(p)}_X(\alpha)>0$.  Indeed this will be the case for any safety loading $\xi<1$.  Thus to obtain a model with a realistic safety loading and have $\psi^{(p)}_X(\alpha)\le 0$, which,  from Section~\ref{RGLT} is necessary to prevent long term exponential growth of the finite time ruin probabilities, the inverse Gaussian process cannot be used.  Instead, from \eqref{rhogt}, for any $\xi>0$, to ensure that $\psi^{(p)}_X(\alpha)\le 0$, where $p$ is given by \eqref{sl01}, one should take a tempered stable process with $\rho\in([1+\xi)^{-1},1)$.  We will thus focus our numerical investigation of the ruin time estimates in \eqref{rft} and \eqref{tulta} on processes satisfying this condition.

This potentially undesirable aspect of the inverse Gaussian process results from the asymptotic (in $t$) behaviour of the asymptotic (in $u$) estimate \eqref{H}, together with the above safety loading considerations. For small values of the initial reserve or over short time periods, exponential growth in the finite time ruin probabilities may not be exhibited.  In this case the inverse Gaussian process with a safety loading of $0.2$ may prove to be an adequate model.
Note however that estimate \eqref{tulta} is not available in this case since $\psi^{(p)}_X(\alpha)>0$.

\subsection{Interpretation of model parameters}
\label{sub:Interp_paras}

We next discuss the role of the various parameters in these models. Fix $a,b>0$ and let $R_t=bY_{at}$.  Then $R$ represents the same aggregate claims process but with different units of currency and time. For example if $Y_t$ is the aggregate claims after $t$ years measured in millions of dollars, and $a=1/4$ and $b=10^{3}$, then $R_t$ is the aggregate claims after $t$ quarters measure in thousands of dollars.  The L\'evy measure of $R$ is
\begin{equation}\label{LMY}
\Pi_R(dx)=\frac{ab^{\rho}ce^{-(\alpha/b)x}\ dx}{x^{1+\rho}},\ \ x>0.
\end{equation}
This is of the same form as $\Pi_Y$ but with different values for the parameters $c$ and $\alpha$.  Thus varying $c$ and $\alpha$ is equivalent to changing the currency and time scale. This is not  the case for  $\rho$, though.  Similarly if we write $p=(1+\xi)EY_1 $, so that $X^{(p)}_t= Y_t- (1+\xi)EY_t$, then
\begin{equation}\label{LMY1}
bX^{(p)}_{at}=R_t- (1+\xi)ER_t,
\end{equation}
and we see that the safety loading also does not depend on the units of currency or time.  Thus, up to a change of scale, the key parameters to vary to obtain different models are $\rho$ and $\xi$.

\section{Numerical Approximation}\label{sec:numerical_approximation}

This section introduces the techniques to be used in numerically approximating the  ruin time distribution, when $X$ is spectrally positive, using \eqref{rft} and \eqref{tulta}.

\subsection{Approximating $B$}\label{}

Expressions for the Laplace transform of $B$ are given in Proposition~\ref{LTofB}, from which $B(t)$ can in principle be evaluated using the Bromwich integral
\begin{equation}
	\label{Bromwich}
	B(t) = \frac{1}{2\pi \rmi} \int_{\epsilon - \rmi \infty}^{\epsilon + \rmi \infty} e^{\gd t} \LT{B}(\gd)\,d\gd,
\end{equation}
where $\epsilon$ is chosen so that $\LT{B}$ in \eqref{HLT} is analytic in the region $\Re(\gd) \ge \epsilon$. An extensive coverage of methods for approximating integrals of this form is presented by Cohen in \cite{Cohen}; in particular, benchmarks comparing the relative errors and computational times for a number of the methods are in \cite{Cohen}.  As a check, we used  two different approaches for the calculation of \eqref{Bromwich}.

\subsubsection*{Fixed-Talbot method}

 The first consists of a straightforward approach by Valk\'o and Abate  \cite{AbateValko05} which relies on multi-precision arithmetic called the fixed-Talbot method; see \cite{Cohen}, p.138.
This method is very attractive from an implementation point of view if one has access to a software package with arbitrary precision arithmetic capabilities (e.g., \emph{Mathematica} \cite{Math}). The fixed-Talbot approach consists of deforming the contour in \eqref{Bromwich} to the path $\gd(\theta) = r \theta (\cot(\theta) + \rmi)$ for $-\pi < \theta < \pi$, where $r$ is a parameter. Integration over this new contour gives
\begin{equation*}
	B(t) = \frac{1}{2 \pi \rmi} \int_{-\pi}^\pi e^{t \gd(\theta)} \LT{B}(\gd(\theta)) \gd'(\theta)\,d\theta.
\end{equation*}
Substituting $\gd'(\theta) = \rmi r (1+\rmi \sigma(\theta))$ with $\sigma(\theta) := \theta +(\theta \cot (\theta) - 1) \cot(\theta)$ and knowing that $B$ is real-valued, we obtain
\begin{equation*}
	B(t) = \frac{r}{\pi}\int_0^\pi \Re\left( e^{t \gd(\theta)} \LT{B}(\gd(\theta))(1+\rmi \sigma(\theta)) \right)\,d\theta.
\end{equation*}
This integral is then approximated using a trapezoidal rule with step size $\pi/M$ and $\theta_j = j \pi/M$ by
\begin{equation*}
B^M(t) = \frac{r}{M}\left(\tfrac12\LT{B}(r) e^{rt} + \frac{r}{M} \sum_{j=0}^{M-1} \Re( (1+\rmi \sigma(\theta_j)) \LT{B}(\gd(\theta_j)) e^{\gd(\theta_j) t})\right).
\end{equation*}
As suggested by Valk\'o and Abate, we choose $r=2M/(5t)$ where $M$ is the number of decimal digits of precision required.

\subsubsection*{Levin method}

The second method provides an alternative for situations where arbitrary precision arithmetic is not available. It starts from the observation that $B(t) = 0$ for $t < 0$ and $B$ is real-valued, so \eqref{Bromwich} simplifies to
\begin{equation}
	\label{BromwichReal}
	B(t) = \frac{2 e^{\epsilon t}}{ \pi} \int_0^\infty \Re(\LT{B}(\epsilon + \rmi u))\cos(u t)\,du.
\end{equation}
This is an integral of oscillatory type so care has to be taken in performing a numerical approximation. We specialised the approach of Levin \cite{Levin82} to our particular case \eqref{BromwichReal},  as follows.

 For $t > 0$ and a large enough $M > 0$ we can approximate \eqref{BromwichReal} by
\begin{equation*}
	B(t) \approx \frac{2}{ \pi} \int_0^M f(u) \cos(tu)\,du
\end{equation*}
where  $f(u):= \Re \LT{B}(\rmi u)$.
Assume $f(u)$ to be of the form
\begin{equation}
	\label{LevinODE}
	f(u) = F'(u) - t F(u) \tan(t u), \quad 0 \le u \le M,
\end{equation}
for some function $F$. Then it follows that
\begin{equation}\label{whatfor}
\begin{aligned}
	 \int_0^M f(u) \cos(tu)\,du
	&=  \int_0^M \left(F'(u)\cos(tu) - t F(u) \sin(tu)\right)\,du\\
	&=  \int_0^M \frac{d}{du}\left( F(u) \cos(tu)\right)\,du\\
	&= F(M) \cos(tM) - F(0).
\end{aligned}
\end{equation}
To find this unknown function $F$, we consider relation \eqref{LevinODE} as an ordinary differential equation (ODE) where $f$ is given and $F$ is to be determined. Since a solution of \eqref{LevinODE} seems difficult to obtain in closed form, we solve the ODE using a numerical method. Assume  there exists $n \in \N$ large enough so that $F$ can be approximated arbitrarily closely on the interval $[0,M]$ by a function $F_n$ given by a linear combination of $n$ predetermined basis functions $p_1(u), \ldots, p_n(u)$:
\begin{equation}\label{F}
F_n(u) = \sum_{k=1}^n c_kp_k(u),
\end{equation}
where $c_1, \ldots, c_n$ are $n$ unknown coefficients to be determined. For simplicity, we made the choice $p_k(u) = T_k(u)$, where $T_k(u)$ is the $k$-th Chebyshev polynomial of the first kind (i.e., $T_k$ satisfies $T_k(\cos(u)) = \cos(k u)$).
Substituting \eqref{F} into \eqref{LevinODE} and using the identity $T_k'(u) = k U_{k-1}(u)$,  where $U_k(u)$ is the Chebyshev polynomial of the second kind, we obtain the following equation
\begin{equation}
\label{system}
f(u) = \sum_{k=1}^n c_k \left(k U_{k-1}(u) - t T_k(u) \tan(tu)\right).
\end{equation}
To find the coefficients $c_1, \ldots, c_n$, we choose $n$ collocation nodes $0=u_1< u_2< \cdots< u_n=M$ and evaluate \eqref{system} at these points in order to set up a system of $n$ equations with $n$ unknowns which can be solved for the $c_i$.  Once these coefficients are obtained, and since  $T_k(0) = \cos(k \pi/2)$,
we have the approximation
\begin{equation}\label{Bapp}
	B(t) \approx \frac{2}{\pi} \left(\sum_{k=1}^n c_k (T_k(M) \cos(t M)- \cos(k \pi/2)) \right).
\end{equation}
Here we see that the approximation depends on good choices of the cut-off value $M > 0$, the number $n$ of basis functions used to approximate $F$, and the location of the collocation nodes $u_1, \ldots, u_n$. After inspection of the behaviour of $\LT{B}$ near zero (see for example Figure~\ref{fig4}), we chose $M=n$ and set the location of the nodes at $u_i = \cot(i\pi/(2n))$, so they accumulate near zero.

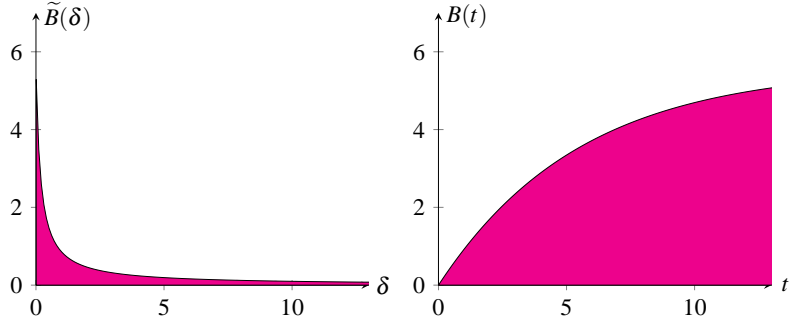
\begin{figure}
\begin{center}
\begin{tikzpicture}[scale=0.9]
\begin{axis}[standard, xmin=0, xmax=13, ymin=0, ymax=7, xlabel=$\delta$, ylabel={$\widetilde B(\delta)$}]
\addplot[filled] table {fig1a.tsv} \closedcycle;
\end{axis}	
\end{tikzpicture}
\begin{tikzpicture}[scale=0.9]
\begin{axis}[standard,xmin=0, xmax=13, ymin=0, ymax=7, xlabel=$t$, ylabel=$B(t)$]
\addplot[filled] table {fig1b.tsv} \closedcycle;
\end{axis}	
\end{tikzpicture}
\end{center}
\caption{\it $\LT B(\gd)$ from \eqref{HLT} and the corresponding $B(t)$ for $t >0$, by numerical inverse Laplace transform. Generated from the tempered stable process $X^{(p)}$ in \eqref{Xp} with $Y$ having triplet \eqref{tempstabletriplet}, where
$\rho=0.99$, $\alpha=1.0$, $c=0.01$, and $\xi=0.2$.}
\label{fig4}
\end{figure}

\subsection{Approximating $\pibar_X^+(u)$ and $P(\tu<\infty)$}\label{u term}

Approximating $\pibar_X^+(u)$ for an arbitrary choice of $u>0$ is straightforward. For example, if $X$ is a tempered stable process then
\begin{equation}
	\label{tail1}
	\overline{\Pi}_X^+(u) = \int_u^\infty \frac{c e^{-\alpha x}}{x^{1+\rho}}\,\rmd x, \quad u > 0
\end{equation}
(see  \eqref{tempstabletriplet}).
The integral here is easily calculated using a numerical quadrature routine, or a simple asymptotic approximation may be adequate for a large enough $u$:
\begin{equation}\label{piasy}
\overline{\Pi}^+_{X}(u)=\frac{c e^{-\alpha u}}{\alpha u^{1+\rho}}(1+o(1))\quad \text{ as }u\to\infty.
\end{equation}

For the models we consider,  $\overline{\Pi}_X^+(u)$ has a singularity at 0,
so it could be expected that the estimate for $P(\tau(u)\le t)$ given by \eqref{rft} is large for small $u$. This suggests that \eqref{tulta} may provide better estimates, at least for small $u$, when applicable.
When $X$ is spectrally positive, it is well-known that the infinite horizon ruin probabilities $P(\tau(u)<\infty)$ can be approximated arbitrarily closely by numerically inverting a Laplace transform; specifically, if $EX_1<0$, then from Theorem 8.1 of \cite{kypbook} we have that
\begin{equation}\label{ktau}
P(\tau(u)<\infty)=1+EX_1W(u)
\end{equation}
where the \emph{scale function} $W$ satisfies
\begin{equation}\label{W}
\int_0^\infty e^{-\gb u}W(u)du=\frac 1{\psi_X(-\gb)}.
\end{equation}
 Recently, an extensive overview of scale functions and their numerical evaluation through the inversion of a Laplace transform has been presented by Kuznetsov, Kyprianou, and Rivero \cite{KKR}.

\section{Simulation Methodology}\label{sim}

Simulation of observations on random variables with tempered stable distributions and the resulting processes has recently become an active topic of research due to their use in a variety of different applications; see the recent survey by Kawai and Masuda \cite{KawaiMasuda11} and the references therein. However, to the best of our knowledge, simulation of finite-time ruin probabilities for tempered stable processes has not yet been covered in the literature. Consequently, we provide some detail on our approaches.

\subsubsection*{Naive approach} A naive approach is tempting; simulate the tempered stable process increment by increment by sampling from the increment distribution and tally the number of paths that pass above level $u$ by time $t$, then calculate the ratio of crossing paths to total paths simulated. When the stability index $\rho < 1$, the law of the process increments can be simulated exactly. On the other hand, when $\rho > 1$ no practical exact simulation method exists and one must resort to an approximation \cite{KawaiMasuda11}. This is particularly troublesome when the probability of ruin is very small due to the possibility of bias in the convergence of the simulation algorithm. In addition, as the simulation of tempered stable random variables is currently not built-in to standard software packages, a custom implementation is needed. Further, simulation of a tempered stable random variable attracts a computational cost greater than that of a stable random variable. These disadvantages motivated us to develop another approach that we shall now explain.

\subsubsection*{Measure change approach}

There is a useful relationship between tempered stable and stable processes that we can exploit.  Consider a spectrally positive stable process $Z = \{Z_t: t \ge 0\}$ of index $\rho\in(0,1)\cup (1,2)$ with characteristic triplet
$(\gamma_Z, 0, \Pi_Z)$, where
\begin{equation}\label{stabtrip}
\gamma_Z = \frac{c}{1-\rho}, \qquad \Pi_Z(\rmd x)=\frac{c\ \rmd x}{x^{1+\rho}},\ \ \ x>0, c>0.
\end{equation}
 It is easily shown, by integration by parts, that for
$\rho\in(0,1)\cup (1,2)$, $\lambda >0$,
\begin{equation}\label{71}
\int_0^\infty(e^{-\lambda x}-1+
\lambda x \mathbf{1}_{\{0<x<1\}})\Pi_Z(\rmd x)=c \lambda ^\rho \Gamma(-\rho)+\frac{c\gl}{1-\rho}.
\end{equation}
Hence the Laplace exponent of $Z_1$ satisfies
\begin{equation*}
\psi_Z(-\lambda ) = \log Ee^{-\lambda  Z_1}= c \Gamma(-\rho)\lambda ^\rho, \quad \lambda  >0.
\end{equation*}
When $\rho\in(0,1)$, the process $Z$ is a pure jump subordinator  while for $\rho\in (1,2)$ it is spectrally positive with finite mean but of unbounded variation.

Let
$\ Z_t^{(p)}=Z_t-pt.$
The characteristics of $X^{(p)}$ and $Z^{(p)}$ are $(\gamma_Y-p, 0, \Pi_Y)$ and $(\gamma_Z-p, 0, \Pi_Z)$
respectively.
Now $Z^{(p)}$ may be obtained from $X^{(p)}$ by an exponential change of measure, i.e. $Z^{(p)}$ is an Esscher transform of $X^{(p)}$. Specifically
assume $X^{(p)}$ and $Z^{(p)}$ are given on a filtered probability space $(\Omega, \sF, (\sF_t)_{t \ge 0}, P)$ and define
 \begin{equation}\label{ET}
 \frac{dQ}{dP}=e^{\alpha X^{(p)}_t-\psi_X^{(p)}(\alpha)t}\ \  \text{ on } \sF_t.
 \end{equation}
 Then $X^{(p)}$ under $Q$ has the same characteristics as $Z^{(p)}$ ; see Theorem 2 of VII.3c in \cite{S}. Rewriting \eqref{ET} as
 \begin{equation}
 dP=e^{-\alpha X^{(p)}_t+\psi_X^{(p)}(\alpha)t} dQ
 \ \  \text{ on } \sF_t
 \end{equation}
we find
 \begin{equation}\begin{aligned}\label{est}
 P(\tau_{X^{(p)}}(u)\le t)&=E_Q(e^{-\alpha X^{(p)}_t+\psi_X^{(p)}(\alpha)t};\tau_{X^{(p)}}(u)\le t)\\
 &=E(e^{-\alpha Z^{(p)}_t};\tau_{Z^{(p)}}(u)\le t) e^{\psi_X^{(p)}(\alpha)t}\\
 &=E(e^{-\alpha Z^{(p)}_t};\tau_{Z^{(p)}}(u)\le t) e^{-(c\Gamma(-\rho)\alpha^\rho+p\alpha)t}
 \end{aligned}\end{equation}
 Thus to calculate ruin probabilities for $X^{(p)}$, we need only simulate the stable process $Z^{(p)}$. This has a number of advantages. First, numerical packages that simulate stable random variables are readily available. Second, the simulation of an increment of a stable process is less computationally expensive than simulating a tempered stable process. And third, the law can be exactly sampled in the case $\rho > 1$.

\subsubsection*{Implementation details}

We construct the stable process $Z^{(p)}$ from a samples of a stable random variable, $S$, as follows. Suppose available a numerical package for simulating a general stable random variable $S\sim \STABLE(\rho, \gb, \mu, \nu)$ with index $\rho \in (0,1) \cup (1,2)$, skewness parameter $\beta$, location parameter $\mu$, and scale parameter $\nu$, having
characteristic exponent expressed in the form (e.g., \cite[Equation (14.24)]{sato})
\begin{equation}\label{CFstablerv}
\Psi_S(\theta) = -\nu |\theta|^\rho \left(1-\rmi \beta \sgn\theta \tan \frac{\pi \rho}{2}\right) + \rmi \mu \theta.
\end{equation}
Then (cf., e.g., proof of \cite[Theorem 14.10]{sato}) the law of $Z^{(p)}_h$, for $h>0$, is obtained from the law of $S$ by setting
\begin{equation}
\beta = 1, \ \mu = -ph,\ \nu = -ch \cos(\pi \rho/2) \Gamma(-\rho).
\end{equation}
We simulate a path $t \mapsto Z_t^{(p)}=Z_t-pt$ by decomposing the path  as a sum of increments over small time intervals $h > 0$;
see Algorithm~\ref{TS2}.

\begin{algorithm}[ht!]
\caption{$P(\tau(u)\le t)$ using time increment $h > 0$ and $n$ simulations} \label{TS2}
\begin{algorithmic}[1]
\State $\text{sum} \gets 0$
\State $\mu \gets -ph$
\State $\nu \gets (-h c \cos(\pi \rho/2) \Gamma(-\rho))^{1/\rho}$
\For{$i = 1, \ldots, n$}
\State  $s \gets 0$, $X \gets 0$
\State $\text{hit} \gets \false$
\While{$s < t$}
\State $dX \sim \STABLE(\rho, 1, \mu, \nu)$
\State $s \gets s + h$
\State $X \gets X + dX$
\If{$X > u$}
\State $\text{hit} \gets \true$
\EndIf
\EndWhile
\If{hit}
\State $\text{sum} = \text{sum} + e^{-\alpha X}$
\EndIf
\EndFor
\State \RETURN\, $(\text{sum} / n) \exp(-(c\,\Gamma(-\rho)\alpha^\rho+p\alpha)t)$

\end{algorithmic}
\end{algorithm}

To get a good approximation of the ruin probability for a given $(u,t)$ when $P(\tau(u) \le t)$ is small, a large number of simulations $n$ and a small time step $h>0$ are needed.  To speed up the simulations, we implemented Algorithm~\ref{TS2} in \texttt{C++11} and modified the algorithm slightly to run in parallel using \texttt{OpenMP}. In the parallelized version of our algorithm, each thread received its own copy of a random number generator (\texttt{mt19937}: Mersenne twister) initialised with independent initial seed. Since the stable distribution is not part of \texttt{C++11} we wrote our own implementation based on an acceptance-rejection method that requires only uniform variates and exponential variates (both available in \texttt{C++11}), Zolotarev's function, and the function $\sinc(x) := \sin(x)/x$ \cite{Devroye09}. This implementation takes the random number generator as an argument so that the generation of variates is independent across threads. The \texttt{for} loop at line $4$ in Algorithm~\ref{TS2} is distributed over the threads using an OpenMP parallel pragma and the shared variable \texttt{sum} is set to be reduced using the addition operator.

In Table~\ref{tab1} and Table~\ref{tab2}, we simulated $P(\tau(u) \le t)$ at $u=0.1$ and $t=2.0$ using the naive and the measure changed algorithm with step sizes of $h=0.0001$ and $h=0.01$, respectively. We also fixed $\rho=0.99$, $c=0.01$, $\alpha = 1$ and $\xi = 0.2$. The tables contain the time\footnote{On an Intel Xeon W3680 @ 3.33Ghz using 12 threads.} in seconds taken to simulate $P(\tau(u) \le t)$ using $n$ sample paths, the mean value of $P(\tau(u) < t)$ calculated using $N=30$ batches of $n$ sample paths, and $\sigma/\sqrt{N}$, where $\sigma$ is the standard deviation of $P(\tau(u) \le t)$ over the $N$ batches.

\begin{table}[h]
\footnotesize
\begin{center}
\begin{tabular}{r|lll|lll}
\hline
 & \multicolumn{3}{|c|}{Naive} & \multicolumn{3}{|c}{Measure Change}\\
$n$ & Time (s) & Mean & $\sigma/\sqrt{N}$ & Time (s) & Mean & $\sigma/\sqrt{N}$ \\
\hline
  32 & 0.079217 & 0.051041667 & 0.0085395638 & 0.076420 &
   0.057849267 & 0.0064226185 \\
 64 & 0.153592 & 0.0484375 & 0.0056237028 & 0.141648 & 0.0543778
   & 0.0038164514 \\
 128 & 0.280040 & 0.054947867 & 0.0048331632 & 0.264577 &
   0.052818533 & 0.0031319083 \\
 256 & 0.558673 & 0.051822867 & 0.0027775502 & 0.526003 &
   0.052552067 & 0.0022664163 \\
 512 & 1.080112 & 0.052604233 & 0.0016828511 & 1.040853 &
   0.052797033 & 0.001440741 \\
 1024 & 2.173619 & 0.051692733 & 0.0011881599 & 2.240990 &
   0.051077267 & 0.00091771839 \\
 2048 & 4.281521 & 0.049755767 & 0.0010048818 & 4.147853 &
   0.050568033 & 0.00064971641 \\
 4096 & 8.624235 & 0.0508301 & 0.00068509617 & 8.315087 &
   0.050364233 & 0.00050435188 \\
 8192 & 17.148903 & 0.0512859 & 0.00049187434 & 16.571840 &
   0.050474467 & 0.00036946037 \\
 16384 & 34.222089 & 0.0512817 & 0.00032080185 & 33.201601 &
   0.050540433 & 0.00029116668 \\
\end{tabular}
\smallskip
\caption{\it Mean of $P(\tau(u) \le t)$ calculated using the two approaches, calculated using $N=30$ batches of $n$ paths and a time step of $h=0.0001$. }\label{tab1}
\end{center}
\end{table}

\begin{table}[h]
\footnotesize
\begin{center}
\begin{tabular}{c|lll|lll}
\hline
 & \multicolumn{3}{|c|}{Naive} & \multicolumn{3}{|c}{Measure Change}\\
$n$ & Time (s) & Mean & $\sigma/\sqrt{N}$ & Time (s) & Mean & $\sigma/\sqrt{N}$ \\
\hline
 32 & 0.004348 & 0.046875 & 0.0069877124 & 0.004592 &
   0.046198033 & 0.0057320205 \\
 64 & 0.005102 & 0.044270833 & 0.0057095406 & 0.004866 &
   0.046170833 & 0.0039138265 \\
 128 & 0.007268 & 0.045833433 & 0.0039947614 & 0.006323 &
   0.047625433 & 0.0028581633 \\
 256 & 0.009335 & 0.046354167 & 0.0021084348 & 0.009371 &
   0.047685367 & 0.0017811631 \\
 512 & 0.015461 & 0.049869767 & 0.0014499284 & 0.014557 &
   0.0495607 & 0.0013563884 \\
 1024 & 0.029704 & 0.0511068 & 0.0011481938 & 0.024450 &
   0.0504672 & 0.00093407293 \\
 2048 & 0.047680 & 0.050992767 & 0.00093047734 & 0.044822 &
   0.050154467 & 0.00067853448 \\
 4096 & 0.090405 & 0.050187133 & 0.00061191755 & 0.093858 &
   0.050157833 & 0.00042720285 \\
 8192 & 0.178851 & 0.050354 & 0.00035644764 & 0.167700 &
   0.050329367 & 0.0003197472 \\
 16384 & 0.344298 & 0.050529 & 0.00026127832 & 0.329710 &
   0.050477933 & 0.00029453508 \\
\hline
\end{tabular}
\smallskip
\caption{\it  Mean of $P(\tau(u) \le t)$ calculated using the two approaches, calculated using $N=30$ batches of $n$ paths and a time step of $h=0.01$.}\label{tab2}
\end{center}
\end{table}

The measure change algorithm performs better than the naive algorithm as it results in a smaller confidence interval (i.e., $\sigma/\sqrt{N}$ is smaller) and shorter run times for the parameter values we are interested in. In some cases, the naive approach can give shorter run times.\footnote{This can occur, for example, when $t$ is large, as the naive approach does not need to simulate the full path up to time $t$ if the process passes above $u$ at some earlier time.} The main advantages of the measure change approach are the ability to accurately simulate the case $\rho > 1$ and the simplicity of implementation as it does not require the simulation of tempered stable random variables.

\section{Comparison of Asymptotic and Simulation Estimates}\label{res}

In this section, we report on the estimates of the ruin time distribution obtained from the asymptotic formula \eqref{rft}, and from the modified version \eqref{tulta}, and compare them to the values obtained from the Monte Carlo simulation as described in Section~\ref{sim}.

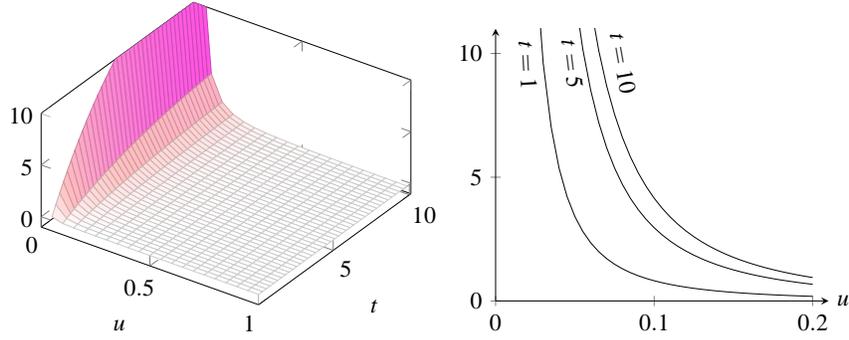
\begin{figure}[h]
\begin{tikzpicture}[scale=0.9]
\begin{axis}[xlabel=$u$, ymax=10.1, ylabel=$t$, xmin=0, zmax=10.0,view={35}{50},unbounded coords=jump]
\addplot3[surf,mesh/ordering=y varies,mesh/rows=40,colormap name=whitehot] file {fig3.tsv};
\end{axis}
\end{tikzpicture}
\begin{tikzpicture}[scale=0.9]
        \begin{axis}[standard, xlabel=$u$, ymin=0, ymax=11, xmin=0,xmax=0.21,xtick={0,0.1,0.2}]
\addplot[clean] table[y index=1] {fig3b.tsv};
\node[above,rotate=-87] at (axis cs:0.01,9.5) {$t=1$};
\addplot[clean] table[y index=2] {fig3b.tsv};
\node[above,rotate=-83] at (axis cs:0.038,9.5) {$t=5$};
\addplot[clean] table[y index=3] {fig3b.tsv};
\node[above, rotate=-80] at (axis cs:0.07,9.5) {$t=10$};
\end{axis}	
\end{tikzpicture}
\caption{\it The value of $\overline{\Pi}_X^+(u) B(t)$ for various $(u,t)$ [left] and for $t=1,5,10$ [right]. Note that values greater than 1 occur. Model and parameter values as in Figure~\ref{fig4}. }
\label{fig5}
\end{figure}

As shown in Figure~\ref{fig5}, it may happen that the leading term in formula \eqref{rft} overestimates the value of $P(\tau(u)\le t)$ for small $u$. Direct application of the asymptotic estimate \eqref{rft} can give putative  values of $P(\tau(u)\le t)$ greater than 1 when $u$ is small. This, of course, is a result of the singularity in $\overline{\Pi}_X^+(u)$ at $u=0$.  However, this is not an issue with the modified estimate \eqref{tulta}, which provides meaningful estimates even for very small $u$.

The value of $B(\infty)$ in \eqref{tulta} can be calculated by combining \eqref{Bn} and \eqref{kapphipsi2}, together with the observation that
$q=|EX^{(p)}_1|$.  This  yields
\begin{equation}
B(\infty)=\frac{\alpha|EX^{(p)}_1|}{(\psi_X^{(p)}(\alpha))^2}.
\end{equation}
Evaluation of $P(\tu<\infty)$ is through \eqref{ktau} and \eqref{W} using the algorithms for numerically inverting Laplace transforms discussed in Section \ref{}; see Figure~\ref{fig7}.
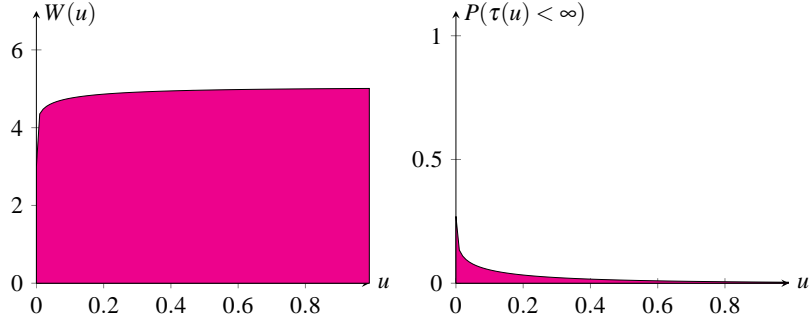
\begin{figure}[h]
\begin{center}
\begin{tikzpicture}[scale=0.9]
\begin{axis}[standard, xmin=0, xmax=0.99, ymin=0, ymax=7, xlabel=$u$, ylabel={$W(u)$}]
\addplot[filled] table {fig4.tsv} \closedcycle;
\end{axis}	
\end{tikzpicture}
\begin{tikzpicture}[scale=0.9]
	\begin{axis}[standard, xmin=0, xmax=0.99, ymin=0, ymax=1.1, xlabel=$u$, ylabel={$P(\tau(u)<\infty)$}]
\addplot[filled] table {fig4b.tsv} \closedcycle;
\end{axis}	
\end{tikzpicture}
\end{center}
\caption
{\it The scale function $u \mapsto W(u)$ in \eqref{W} [left] and $u \mapsto P(\tau(u)<\infty)$ in \eqref{ktau} [right]. Model and parameter values as in Figure~\ref{fig4}.}
\label{fig7}
\end{figure}
Substituting the resulting values into \eqref{tulta} leads to a substantially improved estimate.  To benchmark the estimate \eqref{tulta}, we compared it against the Monte Carlo simulation. Parameter values were chosen so that the resulting model reasonably well approximates practice.   In this regard, we followed Grandell \cite{G} p.145 who cites working actuaries as regarding a safety loading of $0.2$, an initial reserve equal to the expected aggregate claims over a one year time period, and a planning horizon of five years, to be practical.
For added flexibility,  we allowed for an increased  planning horizon of up ten years.

As discussed in Section~\ref{sub:Interp_paras}, the key parameters in the tempered stable models are $\rho$ and $\xi$.  The parameters $c$ and $\alpha$ in  \eqref{tempstabletriplet} correspond to changes of scale. In the scenario illustrated in Figure~\ref{ffig4}, we took a safety loading of $\xi=0.2$.  In order that \eqref{rhogt} hold we then must choose $\rho\in(5/6,1)$.  After some experimentation we found the asymptotic estimate performs better for values of $\rho$ close to 1.  We took $\rho=.99$.  The values of $c$ and $\alpha$ may then be chosen to set a convenient scale in $t$ and $u$. For example in Figure~\ref{ffig4} we took $c=.01$ and $\alpha=1$.  We assume one time unit corresponds to 6 months.  Then the expected aggregate claims over one year is $EY_2=1.9886$.  Thus we plot the ruin probabilities for $0\le t\le 20$ and $0\le u\le 2$.$^5$
\footnote{To illustrate the use of $c$ and $\alpha$ in  setting the scale, if we took $c=(.01)2^{.01}$ and $\alpha=2$ then the plots in Figure~\ref{ffig4} would change only in the scale on the horizontal axes, which would become $0\le t\le 10$ and $0\le u\le 1$.  From Section~\ref{sub:Interp_paras} this corresponds to ruin probabilities for the new process $R_t= X^{(p)}_{2t}/2$ where $X^{(p)}$ is the original process in Figure~\ref{ffig4} with $c=.01$ and $\alpha=1$.  In this case $t$ would now be measured in years, and the expected annual claims would be $ER_1=EY_2/2=0.9943$.}

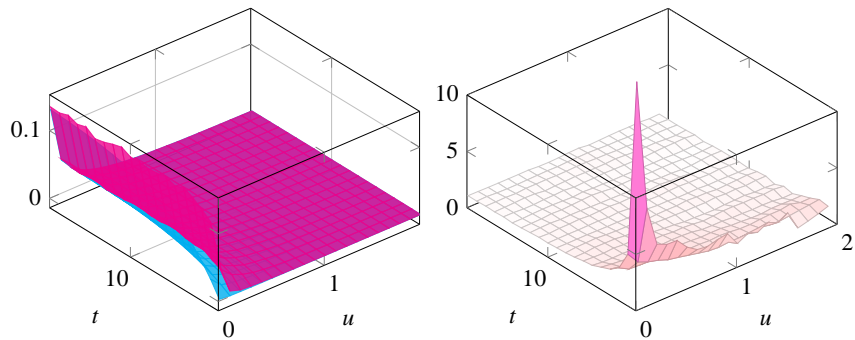
\begin{figure}[h]
\begin{tikzpicture}[scale=0.9]
\begin{axis}[3d box,grid=major,xlabel=$u$,ylabel=$t$,xmin=0,xmax=1.9,view={-40}{50}]
\addplot3[surf,mesh/ordering=y varies,mesh/rows=20,cyan, opacity=0.8, faceted color=cyan] file {fig6b.tsv};
\addplot3[surf,mesh/ordering=y varies,mesh/rows=20,magenta, opacity=0.8, faceted color=magenta] file {fig6a.tsv};
\end{axis}
\end{tikzpicture}
\begin{tikzpicture}[scale=0.9]
\begin{axis}[clip=true,grid=minor,3d box,xlabel=$u$,ylabel=$t$,xmin=0,xmax=2,zmin=0,zmax=10,view={-40}{50}, xtick={0,1,2},ytick={0,10,20}]
\addplot3[surf,mesh/ordering=y varies,mesh/rows=20,colormap name=whitehot] file {fig6c.tsv};
\end{axis}
\end{tikzpicture}
\caption{\it $P(\tau(u) < t)$ for Monte Carlo (magenta) and asymptotic formula (cyan) [left] and the ratio of the two [right].  Model and parameter values as in Figure~\ref{fig4}.}\label{ffig4}
\end{figure}

As shown in Figure~\ref{ffig4}, the estimate \eqref{tulta} performs very well in this case when $t$ is greater than 10 (corresponding to 5 years) and $u$ is greater than $1$ (corresponding to half the expected annual aggregate claims). As Table~\ref{tab3}
indicates, when $u=2$ and $t\ge10$ the relative errors are less than $8.2\%$ and decrease to below $2\%$ by the time $t$ reaches $20$.  Even for $u=1$ the relative error is less than $14\%$ for $t\ge10$.
The variability in the simulated probabilities observed in Table~\ref{tab3} arises because the probability being estimated is very small.  This accords with an insurance company's desire to set its ruin probability over the planning period to be negligible.  In this example it is of order $10^{-3}$ which seems reasonable.

\begin{table}[h]
\footnotesize
\begin{center}
\begin{tabular}{lllllllll}
\hline
 $u$ & $t$ & $a$ & $s$ & $i$ & $a/s$ & $i/s$ & $|a-s|/s$ & $|i-s|/s$ \\
\hline
 1 & 10 & 0.00330802 & 0.003835 & 0.00393118 & 0.862586 & 1.02508 & 0.137414& 0.0250804
   \\
 1 & 12 & 0.00350162 & 0.003829 & 0.00393118 & 0.914499 & 1.02669 & 0.0855013 & 0.0266867 \\
 1 & 14 & 0.00363522 & 0.003872 & 0.00393118 & 0.938849 & 1.01528 & 0.0061151 & 0.0152849 \\
 1 & 16 & 0.00372736 & 0.003725 & 0.00393118 & 1.00063 & 1.05535 & 0.000634196 & 0.0553512 \\
 1 & 18 & 0.00379087 & 0.003704 & 0.00393118 & 1.02345 & 1.06133 & 0.0234522 & 0.0613346 \\
 1 & 20 & 0.00383461 & 0.003971 & 0.00393118 & 0.965654 & 0.989973 & 0.0343456 & 0.0100269 \\
 1.5 & 10 & 0.00127612 & 0.0015 & 0.00151651 & 0.850746 & 1.01101 & 0.149254 & 0.0110093 \\
 1.5 & 12 & 0.0013508 & 0.001506 & 0.00151651 & 0.896947 & 1.00698 & 0.103053 & 0.00698137 \\
 1.5 & 14 & 0.00140234 & 0.001573 & 0.00151651 & 0.891509 & 0.96409 & 0.108491 & 0.0359098
   \\
 1.5 & 16 & 0.00143789 & 0.00138 & 0.00151651 & 1.04195 & 1.09892 & 0.041947 & 0.0989231
   \\
 1.5 & 18 & 0.00146238 & 0.001549 & 0.00151651 & 0.944083 & 0.979028 & 0.0559169 & 0.0209723
   \\
 1.5 & 20 & 0.00147926 & 0.001395 & 0.00151651 & 1.0604 & 1.08711 & 0.0604019 & 0.0871068 \\
 2 & 10 & 0.000543995 & 0.000585 & 0.000646473 & 0.929906 & 1.10508 & 0.0700941 & 0.105081
   \\
 2 & 12 & 0.000575831 & 0.000627 & 0.000646473 & 0.918391 & 1.03106 & 0.0816087 & 0.0310568 \\
 2 & 14 & 0.000597803 & 0.000648 & 0.000646473 & 0.922535 & 0.997643 & 0.0774648 & 0.00235709
   \\
 2 & 16 & 0.000612955 & 0.000574 & 0.000646473 & 1.06787 & 1.12626 & 0.0678655 & 0.126259
   \\
 2 & 18 & 0.000623398 & 0.000642 & 0.000646473 & 0.971025 & 1.00697 & 0.0289752 & 0.00696668
   \\
 2 & 20 & 0.000630592 & 0.00062 & 0.000646473 & 1.01708 & 1.0427 & 0.0170838 & 0.0426978 \\
\hline
\end{tabular}
\smallskip
\caption{\it Comparison and benchmark of asymptotic formula for finite-time ruin probability ($a$) against Monte Carlo simulation ($s$) and infinite-horizon ruin probability ($i$). Simulations were performed with Algorithm~\ref{TS2} using time increment $h=0.001$ and $n=32768$ trials.}\label{tab3}
\end{center}
\end{table}

In Table~\ref{tab3}, the infinite horizon probabilities also provide reasonably good estimates for the simulated probabilities. In general, $P(\tau(u) < \infty)$ can always be used as an upper bound for the probability of ruin in finite time, but there is a question of how precise a bound it may provide. It is quite possible that it may grossly overestimate the finite time ruin  probability.  We observed empirically, for a wide range of parameter values,  that \eqref{tulta} gives a lower bound for the probability of ruin in finite time calculated via simulation, and when it does not, it overestimates only slightly. Thus (13) combined with the infinite horizon ruin probability can be used to place good bounds on the finite time ruin probabilities.

\section{Conclusion}

Up till the publication of the estimate \eqref{rft} in \cite{GM2}, simulation has been the only method of calculating the distribution of the ruin time in the convolution equivalent model. Our aim in the present paper was to show that the function $B(t)$ can be calculated numerically in an interesting and useful class of models, and to examine some of its properties.
The formula \eqref{rft}, though asymptotic, is fast to calculate, and is immediately useful for initial calibration of a model. Once initial estimates of parameters are found in this way, a full Monte Carlo simulation could be performed to refine the estimates if desired. Additionally, the function $B(t)$ in \eqref{H} or its normalised counterpart in \eqref{tulta} can be analysed to provide much insight into the ruin time distribution in these models.  In particular it provides new insight into safely loading management for these models.
Finally we observed empirically that \eqref{tulta} seems to give a lower bound on the probability of ruin in finite time calculated via simulation, which combined with the infinite horizon probability $P(\tau(u) < \infty)$ as an upper bound, might have useful practical applications.

\section{Appendix} \label{app}
Here we prove the  proposition in Section \ref{LTMB}.
\medskip

\begin{proof}[Proof of Proposition \ref{Bprop}]
Assume \eqref{c2} and $Ee^{\alpha X_1}<1$.
Let $\Bbar(t):= B(\infty)-B(t)$, $t\ge 0$,  and use \eqref{HLT} and \eqref{Bn}
 to write
\begin{equation}\begin{aligned}\label{p5}
\int_0^\infty e^{-\lambda t}\Bbar(t) d t
&=
\lambda^{-1}B(\infty) - \int_0^\infty e^{-\lambda t}B(t)d t\\
&=
\lambda^{-1}\left(\frac{q}{-\psi_X(\alpha)\kappa(0,-\alpha)} - \frac{\kappa(\lambda, 0)  }{(\lambda-\psi_X(\alpha)) \kappa(\lambda,-\alpha) }\right).
\end{aligned}\end{equation}
The last expression can be simplified to
\[
\frac{-C_1(\lambda,\alpha) (\kappa(\lambda,0)-q)+\lambda C_2(\lambda,\alpha)
+C_3(\lambda,\alpha)\left(\kappa(\lambda,-\alpha) - \kappa(0,-\alpha) \right)}
{\lambda C_4(\lambda,\alpha)},
\]
where
\[
\lim_{\lambda\dto 0}C_1(\lambda,\alpha)=\lim_{\lambda\dto 0}
(\lambda-\psi_X(\alpha))\kappa(\lambda,-\alpha)=-\psi_X(\alpha)\kappa(0,-\alpha),
\]
\[
\lim_{\lambda\dto 0}C_2(\lambda,\alpha)=\lim_{\lambda\dto 0}\kappa(\lambda, 0) \kappa(\lambda,-\alpha) =q\kappa(0,-\alpha),
\]
\[
\lim_{\lambda\dto 0}C_3(\lambda,\alpha)=\lim_{\lambda\dto 0}-\psi_X(\alpha) \kappa(\lambda, 0)=-q\psi_X(\alpha),
\]
and
\[
\lim_{\lambda\dto 0}C_4(\lambda,\alpha)=\lim_{\lambda\dto 0}-\psi_X(\alpha)   \kappa(0,-\alpha) (\lambda-\psi_X(\alpha))\kappa(\lambda,-\alpha)
=(-\psi_X(\alpha) )^2  \kappa^2(0,-\alpha).
\]
All four limits are finite and strictly positive. Also note that
\[
C_3(0,\alpha)-C_1(0,\alpha)
=  -q\psi_X(\alpha)+\psi_X(\alpha)\kappa(0,-\alpha)
= -\psi_X(\alpha)\left(q- \kappa(0,-\alpha)\right)>0.
\]

Recalling \eqref{kapexp}, we get by monotone convergence
\begin{equation}\label{p00}
\lim_{\lambda\dto 0}
\frac{\kappa(\lambda,0)-q}{\lambda}
= \lim_{\lambda\dto 0}\frac{ \lambda
 \rmd_{{L}^{-1}}+\int_{t\ge 0}\left(1-e^{-\lambda t}\right) \Pi_{{L}^{-1}}(d t)}
 {\lambda}
=\rmd_{{L}^{-1}}+\int_{t\ge 0}t\Pi_{{L}^{-1}}(d t),
\end{equation}
in the sense that both sides are finite or infinite together.
Also observe that
\begin{equation}\begin{aligned}\label{kapd}
\kappa(\lambda,-\alpha) -\kappa(0,-\alpha)
&=
\lambda \rmd_{{L}^{-1}}
+\int_{t\ge0}\left(1-e^{-\lambda t}\right) \Pi_{{L}^{-1}}(d t)
\\
&\qquad\qquad
+\int_{t\ge0} \left(1-e^{-\lambda t}\right)
\int_{h\ge0} \left(e^{\alpha h}-1\right) \Pi_{{L}^{-1}, {H}}(d t, d h).
\end{aligned}\end{equation}
Dividing by $\lambda$ and letting $\lambda\dto 0$ gives
\begin{equation}\begin{aligned}\label{p6a}
\lim_{\lambda\dto 0}
\frac{\kappa(\lambda,-\alpha) -\kappa(0,-\alpha)}{\lambda}
&= \rmd_{{L}^{-1}}+\int_{t\ge0}t\Pi_{{L}^{-1}}(d t)
\\  &\qquad\qquad
+\int_{t\ge0}t
\int_{h\ge0} \left(e^{\alpha h}-1\right) \Pi_{{L}^{-1}, {H}}(d t, d h)
\end{aligned}\end{equation}
again in the sense that both sides are finite or infinite together.
Returning to \eqref{p5},
and letting $\lambda\dto 0$ we find that
\begin{equation}\begin{aligned}\label{c40}
C_4(0,\alpha)\int_0^\infty\Bbar(t) d t
&=
\left( C_3(0,\alpha)-C_1(0,\alpha)\right)\left( \rmd_{{L}^{-1}}+
\int_{t\ge0}t\Pi_{{L}^{-1}}(d t)\right)
\\
&\qquad + C_3(0,\alpha)\int_{t\ge0} t
\int_{h\ge0} \left(e^{\alpha h}-1\right) \Pi_{{L}^{-1}, {H}}(d t, d h).
\end{aligned}\end{equation}
Since $ C_4(0,\alpha)>0$, $ C_3(0,\alpha)>C_1(0,\alpha)$ and $ C_3(0,\alpha)>0$, we find that $\int_0^\infty\Bbar(t) d t<\infty$ if and only if
the two integrals on the righthand side of \eqref{p6a} are finite.

When in addition $X$ is spectrally positive, the second integral is always finite, while if additionally $\sigma_X=0$ the first integral is also.
To see this, we first note  that the integral $\int_{h>1}e^{\alpha h}\Pi_{H}(d h)$ is finite whenever $Ee^{\alpha X_1}<\infty$ by Proposition 7.1 of \cite{G}. Then, treating first the double integral in \eqref{p6a}, we have for the component over $0\le t\le 1$,
\begin{equation}\begin{aligned}\label{p6b}
\int_{0\le t\le 1}t&
\int_{h\ge0} \left(e^{\alpha h}-1\right) \Pi_{{L}^{-1}, {H}}(d t, d h)
\\
&\le
\int_{0\le t\le 1}t\int_{h\ge0}
\left(e^{\alpha }{\bf 1}_{\{0\le h\le 1\}}+e^{\alpha h}{\bf 1}_{\{h>1\}} \right) \Pi_{{L}^{-1}, {H}}(d t)
\\ &\le
 e^{\alpha }\int_{0<t\le 1}t \Pi_{{L}^{-1}}(d t)+
\int_{h>1}e^{\alpha h}\Pi_{H}(d h)< \infty.
\end{aligned}\end{equation}
(Here the first integral on the righthand side is finite, as for any subordinator.)

To deal with the remaining part of the integral over $t>1$, we assume further at this stage that $X$ is spectrally positive.
Thus by
\cite{kypbook},  p.208,
\begin{equation}\label{vsp}
\Pi_{{L}^{-1},{H}}(d t, d h)
=\int_{v\ge 0}\Pi_X(d h+v)P(\whtau_v\in d t)d v,\ t\ge 0, h> 0.
\end{equation}
Since $X_t\to -\infty$ a.s. it follows from Theorem 1 of Doney and Maller \cite{DM} (applied to $-X$) that $E\whtau_1<\infty$.  Thus
\begin{equation}\begin{aligned}\label{v1}
&
\int_{t>1}t\int_{h\ge 0} \left(e^{\alpha h}-1\right)
\Pi_{{L}^{-1},{H}}(d t, d h)\\
&\qquad=
\int_{t>1}t\int_{h> 0} \left(e^{\alpha h}-1\right)
\int_{v\ge 0}\Pi_X(d h+v)P(\whtau_v\in d t)d v\\
&\qquad=
\int_{v\ge0}\int_{h> v} \left(e^{\alpha (h-v)}-1\right)
\Pi_X(d h)\int_{t>1}t P(\whtau_v\in d t)d v\\
&\qquad\le\int_{v\ge0}\int_{h> v} \left(e^{\alpha h}-1\right) e^{-\alpha v}
\Pi_X(d h)E(\whtau_v)d v
\\
&\qquad=E\whtau_1
\int_{h> 0} \left(e^{\alpha h}-1\right) \int_0^h
 ve^{-\alpha v}d v\Pi_X(d h)\\
&\qquad\le
E\whtau_1\int_{0<h\le 1}h \left(e^{\alpha h}-1\right)\Pi_X(d h)
+ E\whtau_1\int_0^\infty ve^{-\alpha v}d v\int_{h>1} e^{\alpha h}\Pi_X(d h).
\end{aligned}\end{equation}
The first integral on the righthand side is obviously finite, and
the second term  on the righthand side is finite
since $Ee^{\alpha X_1}<\infty$.

Finally we deal with the first integral on the righthand side of \eqref{p6a}. We need only consider values of $t>1$.   For this we further assume $\sigma_X=0$.  Since $\rmd_{\whH} >0$ it follows from Corollary 4 of \cite{doneystf} that $\rmd_{H} =0$.  Hence $X$ does not creep up and so by Theorem 3.4 of \cite{GM1}, \eqref{vsp} also holds for $h=0$.  Thus
\begin{equation}\label{vsq}
\Pi_{L^{-1}}(d t)
=\int_{v\ge 0}\pibar_X^+(v)P(\whtau_v\in d t)d v,\ t\ge 0,
\end{equation}
and hence
\begin{equation*}
\int_{t>1}t\Pi_{{L}^{-1}}(d t)
=\int_{v\ge 0}\pibar_X^+(v)\int_{t>1}tP(\whtau_v\in d t)d v
\le
E\whtau_1\int_{v\ge 0}v\pibar_X^+(v)d v<\infty
\end{equation*}
since $E(X_1^+)^2<\infty$ as a result of $Ee^{\alpha X_1}<\infty$. So the last integral converges too.

Thus the two integrals on the righthand side of \eqref{p6a} are finite
and so $\int_0^\infty tB(d t)$ is finite.
\end{proof}

\end{document}